\documentclass[11pt]{amsart}
\usepackage{amsmath,amsfonts,amsthm,amscd,amssymb,mathrsfs,amssymb}
\usepackage{mathrsfs}
\usepackage{graphicx}
\usepackage[all]{xy}
\newtheorem{theorem}{Theorem}

\newtheorem{proposition}[theorem]{Proposition}

\newtheorem{corollary}[theorem]{Corollary}

\newtheorem{remark}[theorem]{Remark}
\newcommand{\CP}{\mathbb{CP}}
\newcommand{\CC}{\mathbb{C}}

\newcommand{\ol}{\overline}
\newcommand{\lra}{\longrightarrow}
\newcommand{\lras}{\,\longrightarrow\,}

\numberwithin{equation}{section}
\numberwithin{theorem}{section}
\begin{document}
\bibliographystyle{alpha} 
\title{On pluri-half-anticanonical system of LeBrun twistor spaces}
\author{Nobuhiro Honda}
\address{Department of Mathematics, Tokyo Institute of Technology, O-okayama,  Tokyo, Japan}
\email{honda@math.titech.ac.jp}
%\author{Jeff Viaclovsky}
%\address{Department of Mathematics, University of Wisconsin, Madison, 
%WI, 53706}
%\email{jeffv@math.wisc.edu}
\thanks{The author was partially supported by the Grant-in-Aid for Young Scientists  (B), The Ministry of Education, Culture, Sports, Science and Technology, Japan. }
\begin{abstract}
In this note, we investigate pluri-half-anticanonical systems on the so called LeBrun twistor spaces. 
We determine its dimension,  the base locus, structure of  the associated rational map, and also structure of general members, in precise form. 
In particular, we show that if $n\ge 3$ and $m\ge 2$, the base locus of the system $|mK^{-1/2}|$  on  $n\CP^2$ consists of two non-singular rational curves, along which any member has singularity, and that if
we blow-up these curves, then the strict transform  of a general members of $|mK^{-1/2}|$ becomes an irreducible non-singular surface.
We also show that if $n\ge 4$ and $m\ge n-1$, then the last surface is a minimal surface of general type with vanishing irregularity.
We also show that the rational map associated to the system  $|mK^{-1/2}|$  is birational if and only if $m\ge n-1$.
\end{abstract}
%\date{February 9, 2009}
\maketitle
\setcounter{tocdepth}{1}
\vspace{-5mm}
%\tableofcontents

%%%%%%%%%%%%%%%%%%%%%%%%%%%%%%%%%%%%%%%%%%%%%%%%
\section{Introduction}
Recently a large number of new  Moishezon twistor spaces were obtained
\cite{Hon-I, Hon-II, Hon-III}.
An important common property of these twistor spaces is that they all admit   $\CC^*$-action.
This implies that the corresponding self-dual metrics admit  isometric $U(1)$-action.
Another common property of the twistor spaces is that the linear system $|K^{-1/2}|$ (so called the fundamental system) is a pencil.
Actually, the structure of the twistor spaces was investigated by making use of {\em reducible}  members of this pencil.
It is also remarkable that no example is known so far of a Moishezon twistor space whose $|K^{-1/2}|$ is empty (or even 0-dimensional).
The author does not know at all whether such a twistor space exists or not,
but it is certain that, if exists, it has to be studied through the linear system $|mK^{-1/2}|$ for some $m\ge 2$, the pluri-half-anticanonical systems.

With this background, a purpose of this note is to give  detailed results on the structure of the pluri-half-anticanonical systems for the most fundamental Moishezon twistor spaces, the {\em LeBrun twistor spaces} \cite[\S 7]{LB91}.
It turns out that it is rather easy to express the dimension of the system $|mK^{-1/2}|$ as a function of $n$ and $m$  in completely explicit form (Proposition \ref{prop:H0}).
The dimension formula immediately implies that, for LeBrun twistor spaces on $n\CP^2$, any element of $|mK^{-1/2}|$ with $m< n-1$,  is a pull-back of a curve of bidegree $(m,m)$ on $\CP^1\times\CP^1$ under the rational map associated to $|K^{-1/2}|$.
It readily follows from this that any element of $|mK^{-1/2}|$ with $m<n-1$ has ordinary singularity of multiplicity $m$ along two rational curves which are the base curves of $|K^{-1/2}|$
(Corollary \ref{cor0}).
The formula also implies that if $m\ge n-1$, general members of $|mK^{-1/2}|$ are {\em not}\, a pull-back of any curve on $\CP^1\times\CP^1$.
However, we can show that even such  members always have  singularity  along the same rational curves (Proposition \ref{prop:main1}, which includes more detailed information on the structure of $|mK^{-1/2}|$).
In particular, they are non-normal.
This is in contrast with a result shown by Pedersen-Poon \cite[Lemma 2.1]{PP94}, which means that any real irreducible members of $|K^{-1/2}|$ is smooth in general.
Also, noting that  the line bundle $K^{-1/2}$ is big for any Moishezon twistor spaces,
this singularity result indicates some peculiar property of LeBrun twistor spaces as compact complex manifolds.
%We note that a proof for the non-projectivity which does not rely on the Hitchin's theorem about the classification of K\"ahlerian twistor spaces \cite{Hi81} is not known (as far as the author knows.)
We also  determine the multiplicity of general members of $|mK^{-1/2}|$ along the base curves precisely (Proposition \ref{prop:main1} (ii)), and also shows that if we blow-up the two base curves, then  the divisor becomes an irreducible, {\em non-singular} surface
(Proposition \ref{prop:main1} (iii)).
Further, we determine structure of the last non-singular surfaces 
(Proposition \ref{prop:main2}).
Finally, we investigate structure of the rational map associated to $|mK^{-1/2}|$.
Thus principal information on the pluri-half-anticanonical system of LeBrun twistor spaces is fully obtained.

\vspace{2mm}\noindent
{\bf Notation.}
For  line bundles $L$ and $L'$, we often write $L+L'$ for the tensor product $L\otimes L'$, and $L^m$ or $mL$ for  $L^{\otimes m}$.
We write $h^i(L)$ for $\dim H^i(L)$.
The notation $\dim |L|$ means $h^0(L)-1$ as usual.
For a complex submanifold $X$ in a complex manifold $Y$,
$N_{X/Y}$ denotes the normal bundle of $X$ in $Y$.
$K$ means the canonical bundle.
The symbol $\sim$ denotes a linear equivalence.

%%%%%%%%%%%%%%%%%%%%%%%%%%%%%%%%%%%%%%%%%%%%%%%%

%%%%%%%%%%%%%%%%%%%%%%%%%%%%%%%%%%%%%%%%%%%%%%%%
\section{Study on pluri-half-anticanonical systems}
%%%%%%%%%%%%%%%%%%%%%%%%%%%%%%%%%%%%%%%%%%%%%%%%
Fix an arbitrary integer $n\ge 3$ and let $Z$ be any LeBrun twistor space on $n\CP^2$, constructed in \cite[\S 7]{LB91}.
We do not put  any assumption for the dimension of the automorphism group for $Z$,
so that the  corresponding self-dual metric (the LeBrun metric) on $n\CP^2$ may admit a 2-torus action, or just a U(1)-action.
Let $F$ be the canonical half of the anticanonical line bundle on $Z$, and $\sigma$ the real structure on $Z$. Recall that the {\em degree} of a divisor or a line bundle on $Z$ is defined as the intersection number with a twistor line.
Then $\deg F=2$, and among all line bundles on $Z$ of positive degree, $F$ can be characterized by the reality (namely $\ol{\sigma}^*F\simeq F$) and  minimality of the degree.
 
 The structure of the complete linear system $|F|$ has been completely understood by 
the works of LeBrun \cite{LB91}, Poon \cite{P92} and Kurke \cite{Ku92}: 
\begin{proposition} Let $Z$ and $F$ be as above. Then we have:
(i) $\dim |F|=3$, (ii) ${\rm{Bs}}\,|F|$ consists of two non-singular rational curves, which are mutually disjoint and conjugate, (iii) a general member of $|F|$ is a non-singular rational surface, which is obtained from $\CP^1\times\CP^1$ by blowing-up $2n$ points.
\end{proposition}
 
 Let $C_1$ and $\ol C_1=\sigma(C_1)$ be the base curve of $|F|$.
 (These curves will play an important role in the following computation.)
Let $\Phi:Z\to\CP^3$ be the rational map associated to the system $|F|$, so that 
 its indeterminacy locus is $C_1\cup \ol C_1$.
 Then basically by the existence of two (mutually conjugate) pencils whose degree is one,
 by which $|F|$ is generated,
 the image of the rational map $\Phi$ is a non-singular quadratic surface.
 Hence the image $\Phi(Z)$ is isomorphic to $\Sigma_0:=\CP^1\times\CP^1$.
   Let $\mu:\tilde Z\to Z$ be the blowing-up along $C_1\cup\ol C_1$, and $E_1$ and $\ol E_1$ the exceptional divisors over $C_1$ and $\ol C_1$ respectively.
Then the composition $\tilde \Phi:={\Phi}\circ\mu$ becomes a morphism (onto $\Sigma_0$), and $E_1$ and $\ol E_1$ become sections of $\tilde{\Phi}$.
By the properties $N_{C_1/Z}\simeq\mathscr O(1-n)^{\oplus 2}$ and   $N_{\ol C_1/Z}\simeq\mathscr O(1-n)^{\oplus 2}$ and an important property that the fiber directions of $\mu|_{E_1}:E_1\to C_1$ and $\mu|_{\ol E_1}:\ol E_1\to \ol C_1$ are {\em different} with respect to the isomorphism $E_1\simeq\ol E_1$ obtained from $\tilde{\Phi}$, 
we have \cite[page 246]{LB91}
\begin{align}\label{nb1}
N_1:=N_{E_1/\tilde Z}\simeq\mathscr O(1-n,-1),\quad
\ol N_1:=N_{\ol E_1/\tilde Z}\simeq\mathscr O(-1,1-n).
\end{align}
(So $\mathscr O_{E_1}(1,0)$ and $\mathscr O_{\ol E_1}(0,1)$ are supposed to be the fiber classes of $\mu|_{E_1}$ and $\mu|_{\ol E_1}$ respectively.)
Thus the situation is summarized in  the following basic diagram:
\begin{equation}\label{cd2}
\xymatrix{
   \tilde Z \ar@{->}[r]^{{\mu}}\ar@{->}[d]_{\tilde\Phi}  & Z  \ar@{->}[dl]^{\Phi}\\
   \Sigma_0. & \\
}
 \end{equation}
Then as $\mathscr O(1)|_{\Sigma_0}\simeq \mathscr O(1,1)$, by the above situation, we have the basic relation 
\begin{align}\label{key1}
\mu^*F\simeq\tilde{\Phi}^*\mathscr O_{\Sigma_0}(1,1)+E_1+\ol E_1.
\end{align}
We first compute the dimension of $H^0(F^m)$:

\begin{proposition}\label{prop:H0}
For any $m\ge 1$, we have the natural isomorphism:
\begin{multline}\label{H0-0}
H^0(Z,F^m)\simeq
H^0(\Sigma_0, \mathscr O(m,m))\\
\oplus 
\left(\bigoplus_{1\le k\le \frac{m}{n-1}}
H^0\left(\Sigma_0, \mathscr O(m-k(n-1),m-k)\oplus \mathscr O(m-k, m-k(n-1))\right)\right).
\end{multline}
In particular we have 
\begin{align}\label{H0-01}
H^0(Z,F^m)\simeq H^0(\Sigma_0, \mathscr O(m,m)),\quad m<n-1.
\end{align}
\end{proposition}

\begin{proof}
By  the Leray spectral sequence for the blow-up $\mu$, we readily obtain 
\begin{align}\label{Leray1}
H^i(Z, F^{ m})\simeq
H^i(\tilde Z,\mu^*F^{ m}),\quad \forall i\ge 0.
\end{align}
To compute the right-hand side, we first notice that by \eqref{key1} 
\begin{align}\notag
 R^q\tilde{\Phi}_*(\mu^*F^{ m})&\simeq
 R^q\tilde{\Phi}_*(\tilde{\Phi}^*\mathscr O(m,m)+mE_1+m\ol E_1))\\
\notag &\simeq \mathscr O(m,m)\otimes R^q\tilde{\Phi}_*\mathscr O(mE_1+m\ol E_1).
\end{align}
Further, as fibers of $\tilde{\Phi}$ are at most a string of smooth rational curves, we have
\begin{align}\notag
R^q\tilde{\Phi}_*\mathscr O(mE_1+m\ol E_1)=0,\quad \forall q>0,\,\,\forall m\ge 0.
\end{align}
Hence the spectral sequence
\begin{align}
E_2^{p,q}=H^p(\Sigma_0, R^q\tilde{\Phi}_*(\mu^*F^{ m}))
\Rightarrow
H^{p+q}(\tilde Z,\mu^*F^{ m})
\end{align}
degenerates at $E_2$-term and we obtain that 
\begin{align}\label{Leray2}
H^i(\tilde Z,\mu^*F^{ m})\simeq
H^i(\Sigma_0, \mathscr O(m,m)\otimes\tilde{\Phi}_*\mathscr O(mE_1+m\ol E_1)),\quad 
\forall i\ge 0.
\end{align}
For the right-hand side, by 
taking the direct image of an obvious exact sequence
\begin{align}\notag
0\lras \mathscr O((k-1)(E_1+\ol E_1))\lras \mathscr O(k(E_1+\ol E_1))\lras 
N_1^{ k}\oplus \ol N_1^{ k}\lras 0,
\end{align}
and computing Ext$^1\,(\simeq H^1)$ by using \eqref{nb1},
we inductively obtain isomorphisms
\begin{align}\label{directim}
\tilde{\Phi}_*\mathscr O(m(E_1+\ol E_1))\simeq \mathscr O\oplus
\left(
\bigoplus_{1\le k\le m} \left(N_1^{ k}\oplus \ol N_1^{ k}\right)
\right), \quad\forall m\ge 0,
\end{align}
where on the right-hand side we are identifying $E_1$ and $\ol E_1$ with $\Sigma_0$ by the morphism $\tilde{\Phi}$.
Therefore by \eqref{Leray1} and \eqref{Leray2} we have 
\begin{align}\label{H0-1}
H^i(Z,F^m)\simeq
H^i(\Sigma_0, \mathscr O(m,m))\oplus 
\left(\oplus_{k=1}^m
H^i(\Sigma_0, N_1^k(m,m)\oplus \ol N_1^k(m,m))\right),\quad
\forall i\ge 0,
\end{align} 
where $N_1^k(m,m)$ means $N_1^{ k}\otimes\mathscr O(m,m)$, and the same for $\ol N_1^{ k}(m,m)$.
By \eqref{nb1} 
we have 
\begin{align}
N_1^{ k}(m,m)\simeq \mathscr O(m-k(n-1),m-k),\quad
\ol N_1^{ k}(m,m)\simeq \mathscr O(m-k, m-k(n-1)).
\end{align}
In particular, if $m-k(n-1)<0$ then  $H^0(\Sigma_0, N_1^k(m,m)\oplus \ol N_1^k(m,m)))=0$. 
Therefore \eqref{H0-1} means the desired isomorphism \eqref{H0-0}.
In particular if $m/(n-1)<1$ namely if $m<n-1$, the second direct summand  does not appear and we obtain \eqref{H0-01}. 
This completes the proof of Proposition \ref{prop:H0}.
\end{proof}

We note that from the proof we have the same isomorphism for $H^1(F^m)$; it suffices to remove the requirement for $k$ in \eqref{H0-0}.

By \eqref{key1}, we have an inclusion $H^0(\Sigma_0,\mathscr O(m,m))\subset H^0(Z,F^m) $ for any $m\ge 1$ from the beginning. 
Hence the isomorphism \eqref{H0-01} implies: 

\begin{corollary}\label{cor0} When $1\le m\le n-2$,  we have the following:  (i)
Any member of the system $|F^m|$ is of the form $\mu(\tilde\Phi^{-1}(D))$ where $D$ is a curve on $\Sigma_0$ whose bidegree is $(m,m)$.
In particular, all members are $\CC^*$-invariant.
(ii) A general member $Y$ of $|F^m|$ is an irreducible non-normal surface which is birational to a ruled surface of genus $(m-1)^2$.
Further, $Y$ has ordinary singularity of multiplicity $m$ along $C_1\cup \ol C_1$.
\end{corollary}

The first statement might be more clearly stated that the system $|F^m|$ is generated by $|F|$
 when $m<n-1$; in other words, $H^0(F^m)={\rm{Sym}}^mH^0(F)$ holds for $m<n-1$.
Thus the situation  becomes apparent when $m<n-1$. 
Next we investigate what happens for the alternative case $m\ge n-1$.

\begin{proposition}\label{prop:main1}
Fix any $m\ge n-1$ and put $l:=[m/(n-1)]$ (the biggest integer not greater than $m/(n-1)$). Let  $\mathscr I\subset\mathscr O_Z$ be the ideal sheaf of the curve $C_1\cup\ol C_1$. Then we have the following:
(i) There is the following sequence of subspaces of $\,H^0(Z,F^m)$:
\begin{multline}\label{subsystem}
\tilde{\Phi}^*H^0(\mathscr O(m,m))
\simeq
H^0(F^m\otimes\mathscr I^m)
\subsetneq
H^0(F^m\otimes\mathscr I^{m-1})
\\\subsetneq
\cdots
\subsetneq
H^0(F^m\otimes\mathscr I^{m-l+1})
\subsetneq
H^0(F^m\otimes\mathscr I^{m-l})
=H^0(F^m).
\end{multline}
Moreover for the differences of the dimensions, we have
\begin{align}\label{diff}
h^0(F^m\otimes \mathscr I^{m-k})
-
h^0(F^m\otimes \mathscr I^{m-k+1})
=
2h^0(\mathscr O(m+k(1-n),\, m-k)).
\end{align}
(ii) Take any $1\le k\le l$ and let $\,Y$ be a general member of the system $|F^m\otimes\mathscr I^{m-k}|$.
Then $Y$ has ordinary singularity of multiplicity $m-k$ along $C_1\cup\ol C_1$, and the strict transform $\,\tilde Y$ of $Y$ into $\tilde Z$ is irreducible and non-singular.
\end{proposition}

\begin{proof}
Let $s_1\in H^0(\mathscr O_{\tilde Z}(E_1+\ol E_1))$ be an element satisfying div\,$(s_1)=E_1+\ol E_1$.
Let $1\le k\le l$ and consider the following injection of sheaves on $\tilde Z$:
\begin{align}\notag
\mu^*F^m - (m-k+1)(E_1+\ol E_1)
\,\stackrel{\otimes s_1}{\lra}\,
\mu^*F^m - (m-k)(E_1+\ol E_1).
\end{align}
Since 
\begin{align}\notag
\left(\mu^*F^m - (m-k)(E_1+\ol E_1)\right)|_{E_1}
&\simeq
\left(\tilde{\Phi}^*\mathscr O(m,m) + k (E_1+\ol E_1)\right)|_{E_1}\\
\notag&\simeq
\mathscr O(m,m)+k\mathscr O(1-n,-1)\\
&\simeq
\mathscr O_{E_1}(m+k(1-n),\,m-k),\label{rest1}
\end{align}
and an analogous isomorphism for the restriction on $\ol E_1$,
we obtain an exact sequence
\begin{multline}\label{key3}
0\lras 
\mu^*F^m - (m-k+1)(E_1+\ol E_1)
\,\stackrel{\otimes s_1}{\lra}\,
\mu^*F^m - (m-k)(E_1+\ol E_1)\\
\lras
\mathscr O_{E_1}(m+k(1-n),\,m-k)
\,\cup\,
\mathscr O_{\ol E_1}(m-k,\,m+k(1-n))
\lras 0.
\end{multline}
Let $\mathscr L_k$ denotes the last non-trivial term of \eqref{key3} (supported on $E_1\cup\ol E_1$).
As for the sign for the integers in $\mathscr L_k$, we readily have
$m+k(1-n)\ge 0$ and $m-k> 0$ under the assumptions $m\ge n-1$ and $1\le k\le l$.
Therefore we have 
\begin{align}\label{key4}
H^1(\mathscr L_k)=0 \quad(1\le k\le l).
\end{align}
On the other hand, if $k=1$, the first non-trivial term of \eqref{key3} becomes isomorphic to 
$\tilde{\Phi}^*\mathscr O(m,m)$ by \eqref{key1}, so that its $H^1$ also vanishes.
Therefore,
 the cohomology exact sequence for \eqref{key3} in the case $k=1$ gives an exact sequence
\begin{align}\label{exact1}
0
\lras
 H^0(\tilde{\Phi}^*\mathscr O(m,m))
\,\stackrel{\otimes s_1}{\lra}\,
H^0(\mu^*F^m - (m-1)(E_1+\ol E_1))
\lras
H^0(\mathscr L_1)
\lras
0
\end{align}
and also the vanishing (from \eqref{key3} with $k=1$) 
\begin{align}
H^1(\mu^*F^m - (m-1)(E_1+\ol E_1))=0.
\end{align}
Therefore by \eqref{key4}, we inductively obtain the exact sequence
\begin{multline}\label{key5}
0\lras
H^0(\mu^*F^m - (m-k+1)(E_1+\ol E_1))
\,\stackrel{\otimes s_1}{\lra}\,
H^0(\mu^*F^m - (m-k)(E_1+\ol E_1))\\
\lras
H^0(\mathscr L_k)
\lras 0, \quad(1\le k\le l),
\end{multline}
and the vanishing
\begin{align}
H^1(\mu^*F^m - (m-k)(E_1+\ol E_1))=0\quad (1\le k\le l).
\end{align}
By using the exact sequences \eqref{key5} successively, we obtain 
\begin{align}
h^0(F^m\otimes\mathscr I^{m-l})
&=h^0(F^m\otimes\mathscr I^{m-l+1})+h^0(\mathscr L_l)\\
&=h^0(F^m\otimes\mathscr I^{m-l+2})+h^0(\mathscr L_{l-1})+h^0(\mathscr L_l)\notag\\
&=\cdots\notag\\
&=h^0(F^m\otimes\mathscr I^m)+\sum_{1\le k\le l}h^0(\mathscr L_k)\notag\\
&=h^0(\mathscr O(m,m))+\sum_{1\le k\le l}h^0(\mathscr L_k).\notag
\end{align}
This directly gives the sequence \eqref{subsystem} and the dimension formula \eqref{diff}. Hence we obtain (i).

For the second claim, we first note that,  since $|\mathscr O(m,m)|$ (on $\Sigma_0$) is base point free,
by the injection in the exact sequence \eqref{exact1}, we obtain  ${\rm{Bs}}\,|\mu^*F^m-(m-1)(E_1+\ol E_1)|\subset E_1\cup \ol E_1$.
Further, as $\mathscr L_1=\mathscr O(m+1-n,m-1)\cup\mathscr O(m-1,m+1-n)$ and $m+1-n\ge 0$ and $m-1> 0$ as  already remarked,
we have ${\rm{Bs}}\,|\mathscr L_1|=\emptyset$.
Therefore the surjectivity of the restriction map in \eqref{exact1} implies 
 ${\rm{Bs}}\,|\mu^*F^m-(m-1)(E_1+\ol E_1)|=\emptyset$.
Once this is obtained, again by using the exact sequences \eqref{key5} and verifying 
${\rm{Bs}}\,|\mathscr L_k|=\emptyset$ for $1\le k\le l$, we inductively obtain 
$$
{\rm{Bs}}\,\left|\mu^*F^m-(m-k)(E_1+\ol E_1)\right|=\emptyset,\quad 1\le k\le l.
$$
Therefore by Bertini's theorem a general member of the system $|\mu^*F^m-(m-k)(E_1+\ol E_1)|$ is irreducible and non-singular (when $1\le k\le l)$.
Let $\tilde Y$ be such a member.
Then if we put $Y:=\mu(\tilde Y)$,
since $\tilde Y|_{E_1}\in |\mathscr O(m+k(1-n),m-k)|$ and 
$\tilde Y|_{\ol E_1}\in |\mathscr O(m-k,m+k(1-n))|$, 
taking into account for the directions of the blowing-down $\mu$
(see \eqref{nb1}),
 this gives a general member of the system $|F^m\otimes \mathscr I^{m-k}|$,
which has ordinary singularity of multiplicity $m-k$ along $C_1\cup\ol C_1$.
Thus   we obtain the  claim of (ii),
and finish a proof of Proposition \ref{prop:main1}.
\end{proof}

\begin{remark}
{\em
It is possible to give generators of the system $|F^{n-1}|$ in completely explicit form.
They can be chosen in such a way that all irreducible components of the generators are  degree 1 divisors.
It might be interesting to remark that the situation is almost the same for the twistor spaces studied in \cite{Hon-II}; namely LeBrun twistor spaces and the twistor spaces in \cite{Hon-II} have a common property, at least for the linear system $|F^{n-1}|$. (See \cite[Lemma 2.8]{Hon-II}; also compare Proposition \ref{prop:main2} below with \cite[Theorem 4.3]{Hon-II}.)
}
\end{remark}

As an immediate consequence of Proposition \ref{prop:main1} we obtain

\begin{corollary}\label{cor:sing}
For any LeBrun twistor spaces on $n\CP^2$ with $n\ge 3$, a divisor belonging to the system $|F^m|$ can be non-singular only when $m=1$.
\end{corollary}

Recall that by a result of Pedersen-Poon \cite{PP94}, on any twistor spaces on $n\CP^2$,
all real irreducible members of the system $|F|$ are non-singular.
Corollary \eqref{cor:sing} shows that this is true only for the system $|F|$ itself.
We also remark that when $n=2$, Proposition \ref{prop:main1} and Corollary \ref{cor:sing} do not hold, since in that case ${\rm{Bs}}\,|F|=\emptyset$ (\cite[Proposition 2.6]{P86}). 

\begin{remark}{\em
One might think that since the restriction of $F^m$ to a non-singular member $S\in |F|$ is isomorphic to $K_S^{-m}$, and since any member of $|K_S^{-m}|$ is readily seen to contain the curve $C_1$ and $\ol C_1$ by a high multiplicity, any member of $|F^m|$ would have singularity along $C_1\cup\ol C_1$. But of course this argument just shows that the restriction $Y|_S$ (where $Y\in |F^m|$) contains $C_1\cup\ol C_1$ as non-reduced components, and does not show that $Y$ itself has singularity along $C_1\cup\ol C_1$.
}
\end{remark}
%We also remark that Proposition \ref{prop:main1} also has the following implication:
%\begin{corollary}\label{cor1}
%If $n\ge 3$, any LeBrun twistor space on $n\CP^2$ cannot be a projective algebraic manifold.
%\end{corollary}

%\begin{proof}
%Suppose that a LeBrun twistor space  has a very ample line bundle $L$.
%Then $L+\ol{\sigma}^*L$ is also a very ample line bundle.
%Hence by Bertini's theorem a general element of $|L+\ol{\sigma}^*L|$ is non-singular.
%On the other hand by reality we have  $L+\ol{\sigma}^*L\simeq F^m$ for some $m\ge 1$.
%By taking self tensor product if necessary, we may suppose $m\ge 2$.
%By Corollary \ref{cor0} and Propositions \ref{prop:main1}, a general member of $|F^m|$ has singularities along the curve $C_1\cup \ol C_1$, for any $m\ge 2$.
%This is a contradiction.
%Hence  there are no very ample line bundles.
%\end{proof}

%By a theorem of Hitchin \cite[Theorem 6.1]{Hi81}, Corollary \ref{cor1} is true also for $2\CP^2$.
%But the above argument does not work 
%since Corollary \ref{cor0} and Proposition \ref{prop:main1} are valid only when $n\ge 3$ as above.

Next we investigate 
the structure of the strict transform of a general member of $ |F^m|$ in the case $m\ge n-1$.
Because structure of members of the smallest subsystem $|F^m\otimes\mathscr I^{m}|$ in the sequence \eqref{subsystem} is already clear as in Corollary \ref{cor0} (ii), we consider the case $Y\in |F^m\otimes\mathscr I^{m-k}|$ with $1\le k\le l$.
(Recall that $|F^m\otimes \mathscr I^{m-l}|=|F^m|$.)

\begin{proposition}\label{prop:main2}
Let $Z$ be any LeBrun twistor spaces on $n\CP^2$ ($n\ge 3$), $m\ge n-1$ an integer, $l=[m/(n-1)]$,  $Y$ a general member of the system $|F^m\otimes \mathscr I^{m-k}|$ with $1\le k\le l$, and $\tilde Y$ the strict transform of $\,Y$ into $\tilde Z$ as before.
Then we have:
(i) The restriction $\tilde{\Phi}|_{\tilde Y}:\tilde Y\to\Sigma_0$ is surjective and its degree is $2k$.
(ii) If $n\ge 4$, or if $n=3$ and $m>2$, then $\tilde Y$ is a (non-singular) minimal surface of general type with vanishing irregularity.
(iii) If $n=3$ and $m=2$, $\tilde Y$ is a (non-singular) K3 surface.

\end{proposition}

\begin{proof}
(i) is immediate from $\tilde Y\sim \tilde{\Phi}^*\mathscr O(m,m)+k(E_1+\ol E_1)$, recalling
that $E_1$ and $\ol E_1$ are sections of $\tilde{\Phi}$ and that the image of $\tilde\Phi^*|\mathscr O(m,m)|$ is strictly smaller than $|F^m|$ when $m\ge n-1$ and $1\le k\le l$ by Proposition \ref{prop:H0} or Proposition \ref{prop:main1}.
For (ii) and (iii),
by adjunction formula, we have $K_{\tilde Y}
\simeq 
(K_{\tilde Z}+\tilde Y)|_{\tilde Y}$.
Further we have
\begin{align}\notag
K_{\tilde Z}+\tilde Y
&\simeq
\mu^*K_Z+(E_1+\ol E_1)+\tilde Y\\
\notag&\simeq
\mu^*(F^{-2})+(E_1+\ol E_1)+\mu^*F^m-(m-k)(E_1+\ol E_1)\\
\notag&\simeq
\mu^*(F^{m-2})-(m-k-1)(E_1+\ol E_1)\\
&\simeq
\tilde{\Phi}^*\mathscr O(m-2,m-2)+(k-1)(E_1+\ol E_1)\qquad(\text{by \eqref{key1}}).
\label{can1}
\end{align}
Therefore, since we have $m-2>0$  if $n\ge 4$, or $m>2$ if $n=3$, and since $l-1\ge 0$,  we obtain that $h^0(\nu K_{\tilde Y})$ grows quadratically as a function of $\nu$.
Hence  $\tilde Y$ is a surface of general type.
On the other hand if $n=3$ and $m=2$, we have $l=1$ and hence \eqref{can1} means that $K_{\tilde Y}$ is trivial. 
Next we prove that $H^1(\mathscr O_{\tilde Y})=0$ for both cases of (ii) and (iii).
For this, by the standard exact sequence 
$0\lras \mathscr O_{\tilde Z}(-\tilde Y)
\lras \mathscr O_{\tilde Z}
\lras  \mathscr O_{\tilde Y}
\lras 0$, it suffices to show $H^2(\mathscr O_{\tilde Z}(-\tilde Y)
)=0$.
By duality, the last space is the dual of $H^1(K_{\tilde Z}+\tilde{Y})$.
Moreover, by \eqref{can1}, we obtain, with the aid of \eqref{directim}
\begin{align}
H^1(K_{\tilde Z}+\tilde{Y})
&\simeq
H^1\left(\mathscr O(m-2,m-2)\otimes(
\mathscr O\oplus
(N_1\oplus \ol N_1)
\oplus
\cdots
\oplus
(N_1^{k-1}\oplus \ol N_1^{k-1})
\right).
\end{align}
As $N_1\simeq\mathscr O(1-n,-1)$, in order to prove $H^1(K_{\tilde Z}+\tilde Y)=0$, it is enough to show that, if we write $\mathscr O(m-2,m-2)\otimes N_1^{k-1}\simeq \mathscr O(a,b)$, then $a\ge 0$ and $b\ge 0$ hold.
For this, putting $m=l(n-1)+q$ with $0\le q<n-1$ as before, we compute
\begin{align}\notag
a= m-2+(k-1)(1-n)
&\ge
l(n-1)+q-2+(l-1)(1-n)
=n+q-3\ge 0,
\end{align}
and
\begin{align}
b=m-k-1
\ge
 l(n-1)+q-l-1
 =
 l(n-2)+q-1\ge q\ge 0.
\end{align}
Hence we obtain $H^1(K_{\tilde Z}+\tilde Y)=0$, so that $H^1(\mathscr O_{\tilde Y})=0$.
In particular, in the situation of (iii),  $\tilde Y$ is a $K3$ surface.

To complete a proof of Proposition \ref{prop:main2},
it remains to prove that $\tilde Y$ does not contain a $(-1)$-curve, in the situation of (ii).
Since $(\tilde{\Phi}^*\mathscr O(m-2,m-2))\cdot D\ge 0$ for any curve $D$ on $\tilde Y$, 
by \eqref{can1} we have  $K_{\tilde Y}\cdot D\ge 0$ for any curve $D\subset \tilde Y$ unless $D\subset E_1\cup \ol E_1$.
So to prove that there is no $(-1)$-curve, it suffices to show that if we take a sufficiently general $Y$,  then the restriction
$\tilde Y|_{E_1}$  does not contain  a $(-1)$-curve.
For this, by  \eqref{rest1} we have $\tilde Y|_{E_1}\in |\mathscr O(m+k(1-n),m-k)|$.
Further, we have 
\begin{align}
m+k(1-n)
=l(n-1)+q+k(1-n)=(l-k)(n-1)+q\ge q
\end{align}
and 
\begin{align}
m-k
=
l(n-1)+q-k\ge l(n-1)+q-l=l(n-2)+q\ge 2.
\end{align}
Therefore (by the surjectivity of the restriction map in the exact sequence \eqref{key5}),
if we take sufficiently general $ Y\in |F^m\otimes \mathscr I^{m-k}|$, the curve $\tilde Y|_{E_1}$ is a non-singular irreducible curve of bidegree $(a,b)$ with $a\ge q$ and $b\ge 2$. 
Hence if $q>1$, $a>1$ follows and hence $\tilde Y|_{E_1}$ cannot be a rational curve for sufficiently general $Y$.
On the other hand, if $a=0$, we obtain $l=k$ and $q=0$, and hence $\tilde Y|_{E_1}$ becomes a curve of bidegree $(0,l(n-2))$, so that it is a union of $l(n-2)$ smooth rational curves of bidegree $(0,1)$, for sufficiently general $Y$.
Now since $\tilde Y$ intersects $E_1$ transversally along these rational curves, 
the self-intersection numbers of these curves in $\tilde Y$ is 
$$
(0,1)\cdot N_{E_1/\tilde Z}=\mathscr O(0,1)\cdot \mathscr O(1-n,-1)=1-n<-2.
$$
Therefore any of  the $l(n-2)$ rational curves cannot be a $(-1)$ curve on $\tilde Y$.
Finally, if $a=1$, $l=k$ and $q=1$ follow and the curve $\tilde Y|_{E_1}$ becomes a non-singular rational curve of bidegree $(1, l(n-2)+1)$. However we have
\begin{align}\notag
\mathscr O(1, l(n-2)+1)\cdot N_{E_1/\tilde Z}&=
-l(n-1)(n-2)-n\le -1\cdot 2\cdot 3-4=-10.
\end{align}
This means that the self-intersection number of  $\tilde Y|_{E_1}$ in $\tilde Y$ is less than $-9$.
Hence it cannot be a $(-1)$-curve, too.
By the same reason, any curve $D\subset \tilde Y\cap \ol E_1$ cannot be a $(-1)$-curve for general $Y\in |F\otimes\mathscr I^{m-k}|$.
Thus we have proved that $\tilde Y$ does not contain a $(-1)$-curve.
Namely it is a minimal surface.
\end{proof}

%When $k=1$ (namely when  $Y\in |F^m\otimes \mathscr I^{m-1}|$) the restriction $\tilde\Phi|_{\tilde Y}\to \Sigma_0$ becomes generically $2:1$. So in this case it is easy to compute other basic invariants of $\tilde Y$:

%\begin{proposition}
%Let $n\ge 4$ and $m\ge n-1$, $Y\in |F^m\otimes\mathscr I^{m-1}|$ a general member, and $
%\tilde Y$ the strict transform, which becomes non-singular minimal surface of general type by Propositions \ref{prop:main1} and  \ref{prop:main2}. Then we have:
%\begin{align}
%K_{\tilde Y}^2=4(m-2)^2,\quad p_g(\tilde Y)=(m-1)^2.
%\end{align}
%\end{proposition}

%It is also possible to compute $K_{\tilde Y}^2$ for any $\tilde Y$ in Proposition \ref{prop:main2} (ii).

It is also possible to compute $K_{\tilde Y}^2$ for $\tilde Y$ in Proposition \ref{prop:main2} (ii).

Finally we show the following result about the rational map associated to $|F^m|$:

\begin{proposition}
Let $Z$ be any LeBrun twistor space on $n\CP^2$ with $n\ge 3$, and $\Phi_m$ the rational map associated to the linear system $|F^m|$ on $Z$. Then we have the following.
(i) If $\,1\le m\le n-2$, the rational map  $\Phi_m$ factors as $g_m\circ\Phi_1$, where $g_m$ is a holomorphic map from $\Sigma_0$ associated to the linear system $|\mathscr O(m,m)|$.
(In particular, the image $\Phi_m(Z)$ is biholomorphic to $\Sigma_0$.)
(ii) If $m\ge n-1$, the map $\Phi_m$ is birational over its image.
\end{proposition}

\begin{proof}
(i) is obvious from \eqref{H0-01} or Corollary \ref{cor0}.
%For (ii), let $t\in H^0(F)$ be any non-zero element.
%Then since we can define an injection $\otimes t^{m-n+1}:H^0((n-1)F)\to H^0(F^m)$ that induces a surjection $\pi:\mathbb P H^0(F^m)^*\to \mathbb P H^0((n-1)F)^*$ satisfying $\Phi_{n-1}=\pi\circ\Phi_m$, it suffices to show the claim for the case $m=n-1$.
For (ii), by Propositions \ref{prop:main1} and  \ref{prop:main2}, the strict transform $\tilde Y$ of a general member $Y\in |F^{n-1}|$ is a non-singular surface in $\tilde Z$ whose restriction of $\tilde{\Phi}$ is a surjection of degree $2l$.
Therefore for a general fiber $f$ of $\tilde{\Phi}:Z\to \Sigma_0$ (which is non-singular rational curve), $\mathscr O(\tilde Y)|_f\simeq\mathscr O(2l)$.
Therefore, recalling that $\tilde{\Phi}$ is originally associated to $|F|$, it follows that the system $|\tilde Y|$ on $\tilde Z$ induces a birational map.
Hence the same is true for the system $|Y|=|F^m|$ on $Z$, as required.
\end{proof}

\end{document}